\documentclass{article}
\usepackage{amssymb}
\usepackage{amsmath}
\usepackage{amsthm}

\theoremstyle{plain}
\newtheorem{theorem}{Theorem}[section]
\newtheorem{lemma}[theorem]{Lemma}
\newtheorem*{cor*}{Corollary}
\theoremstyle{remark}
\newtheorem*{remarks*}{Remarks}
\newtheorem*{notedef*}{Notation and definitions}
\newtheorem*{notedefr*}{Notation and definitions regarding stabilizers
and supports}

\newcommand{\notarrow}{\kern .42em\not\kern -.42em\longrightarrow}
\newcommand{\sq}[1]{\ensuremath{\langle #1 \rangle}}
\newcommand{\setof}[2]{\{\,#1\,:\,#2\,\}}

\newcommand{\setmin}{\smallsetminus}
\newcommand{\forces}{\Vdash}
\newcommand{\QQ}{\mathbf{Q}}
\newcommand{\FF}{\mathbb{F}}
\newcommand{\NN}{\mathcal{N}}
\newcommand{\MM}{\mathcal{M}}
\newcommand{\II}{\mathcal{I}}

\newcommand{\statepc}[1]{PC($\aleph_0$,$#1$)}
\newcommand{\statec}[1]{C($\aleph_0$,$#1$)}

\newcommand{\pr}{\mathrm{pr}}
\newcommand{\dd}{d}
\newcommand{\logp}{\mathrm{log}_{*p}}

\DeclareMathOperator{\Orb}{{\rm Orb}}
\DeclareMathOperator{\Aut}{{\rm Aut}}
\DeclareMathOperator{\Dom}{{\rm Dom}}
\DeclareMathOperator{\Ran}{{\rm Range}}
\DeclareMathOperator{\Span}{{\rm Span}}

\begin{document}
\date{}
 
\title{Partial choice functions for families of finite sets}

\author{
Eric J.~Hall\\
\and
Saharon Shelah\footnote{The second author's research was supported by
the US-Israel Binational Science Foundation.
Publication  934.}
}

\maketitle
{
\renewcommand\thefootnote{}
\footnotetext{Research partially supported by NSF grant
No.\ NSF-DMS 0600940.}
}

\begin{abstract}
Let $p$ be a prime.
We show that ZF $+$ ``Every countable
set of $p$-element sets has an infinite partial choice function''
is not strong enough to prove that every countable set of $p$-element
sets has a choice function, answering an open question from
\cite{DHHKR08}.
The independence result is obtained by way of a permutation
(Fraenkel-Mostowski) model in which the set of atoms
has the structure of a vector space over the field of $p$
elements, and then the use of atoms is eliminated
by citing an embedding theorem of Pincus.
By way of comparison, some
simpler permutation models are considered in which
some countable families of $p$-element sets fail to have
infinite partial choice functions.
\end{abstract}

\begin{section}{Introduction} \label{intro}

Let \statec{n}\ be the statement asserting that every
infinite, countable set of $n$-element sets has a choice function.
Let \statepc{n}\ be the statement asserting that every
infinite, countable set $C$ of $n$-element sets
has an infinite partial choice function.
That is, \statepc{n}\  asserts that there is
a choice function whose domain is an infinite subset of $C$.
(Recall \statec{n}\ is Form 288($n$), and
\statepc{n}\ is Form 373($n$) in Howard and Rubin's reference
\cite{HR98}.  Also, \statec{2} is Form 30, and \statepc{2} is Form 18.)
The main result of this paper is that for any prime $p$, 
\statepc{p}\ does not imply \statec{p}\ in ZF.  This answers 
questions left open in 
\cite{DHHKR08}.

The independence results are obtained using the technique of
permutation models (also known as Fraenkel-Mostowski models).  
See Jech \cite{Jech73} for basics about permutation models
and the theory ZFA (ZF modified
to allow atoms).
A suitable permutation model will establish 
the independence of 
\statec{p}\ from \statepc{p}\ in the context of ZFA.
This suffices by work of Pincus \cite{Pincus72}
(extending work of Jech and Sochor), which shows that
once established under ZFA, the independence result transfers to
the context of ZF
(this is because the statement \statepc{p}\ is ``injectively boundable''; 
see also Note 103 in \cite{HR98}).

Section~\ref{main} is the proof of the main result,
Theorem~\ref{theorem:main}.  
Readers with some experience with permutation models may 
wonder whether 
the model used in the proof of Theorem~\ref{theorem:main}
is unnecessarily complicated. 
Section~\ref{simpler}
explains why certain simpler models which may appear promising
candidates to
witness the independence of \statepc{2}\ from \statec{2}\ in
fact fail to do so.

\end{section}
\begin{section}{The main theorem} \label{main}

\begin{theorem}\label{theorem:main}
Let $p$ be a prime integer.
In ZF, \statepc{p}\ does not imply \statec{p}\
\end{theorem}
\begin{proof}
As discussed in the
Introduction, it suffices describe a
permutation model in which \statepc{p}\ holds
and \statec{p}\ fails.  
Let $\MM$ be a model of ZFAC whose set of atoms is countable and infinite;
we will work in $\MM$ unless otherwise specified.  
We will describe a permutation submodel 
of $\MM$.

First, we set some notation for a few vector spaces over 
the field $\FF_p$ with $p$ elements.
Let $W=\oplus_{i=1}^\infty \FF_p$, so each
element of $W$ is a sequence $w=(w_1,w_2,w_3,\dots)$ of elements
of $\FF_p$, with
at most finitely many nonzero terms.  Let
$G$ be the full product $\otimes_{i=1}^\infty \FF_p$ (sequences may have
infinitely many nonzero elements).  Finally, let $U=\FF_p\times W$,
so each element
of $U$ is a pair $(a,w)$ with $a\in\FF_p$, $w\in W$.

For each $w\in W$, let $U_w=\setof{(a,w)}{a\in\FF_p}$, so that
$\mathcal P=\setof{U_w}{w\in W}$ is a partition of $U$ into 
sets of size $p$.  We define an action such that each $g\in G$ gives
an automorphism of $U$, and such that the $G$-orbits are
the elements of the partition $\mathcal P$, as follows.
For each $(a,w)\in U$ and $g\in G$, let
\[
(a,w)g = (a+\sum_{i=1}^\infty w_ig_i, w)
\]
(where $w_i$ is the $i^{\mathrm{th}}$ entry in the sequence $w$, and
likewise $g_i$; 
the product $w_ig_i$ is in the field $\FF_p$, and the
sum $a+\sum_i w_ig_i$ is a (finite) sum in $\FF_p$).
This action induces an isomorphism of $G$ with
a subgroup of $\Aut(U)$; 
we will henceforth identify $G$ with this subgroup,
think of the operation on $G$ as composition instead of addition, and
continue to let $G$ act on the right.

\begin{remarks*}
It is clear from the original definition of $G$ that $G$ is
abelian, and all its non-identity elements have order $p$.  
As a subgroup of $\Aut(U)$, $G$ may be characterized as the 
group of all automorphisms of $U$ which act on each element of the
partition $\mathcal P$ and have order $p$ or $1$.  Equivalently, $G$ is the
group of all automorphisms of $U$ which act on each element of
$\mathcal P$ and act trivially on $U_{\mathbf{0}}$.
\end{remarks*}

Now, 
identify the set of atoms in $\MM$ with the vector space $U$.
Thus, we think of each $g$ in 
$G$ as a permutation of the set of atoms.  Each permutation of $U$
extends uniquely to an automorphism of $\MM$, and so we will also think
of $G$ as a subgroup of $\Aut(\MM)$.  

Let $\II$ be a (proper) ideal on $W$ such that
\begin{enumerate}
\setlength{\itemsep}{0ex}
\item[($*$1)] every infinite subset
of $W$ contains an infinite member of $\II$, and
\item[($*$2)] 
$A\in\mathcal I\Rightarrow\Span(A)\in\II$,
\end{enumerate}
where $\Span(A)$ is the $\FF_p$-vector subspace of $U$ 
generated by $A$. 
For proof of the existence of
such an ideal,  see Lemma~\ref{lemma:ideal}.

\begin{notedefr*}
For $A\subset W$ and $g\in G\subset\Aut(\MM)$, we say 
``$g$ \emph{fixes at} $A$'' if $g$ fixes each atom in
$\FF_p\times A=\bigcup_{w\in A}U_w$.  
Let $G_{(A)}$ denote the subgroup of $G$ consisting of elements which
fix at $A$ (i.e., $G_{(A)}$ is the pointwise stabilizer of 
$\bigcup_{w\in A}U_w$).  When $A=\{a_1,\dots,a_n\}$ is finite,
we may write $G_{(a_1,\dots,a_n)}$ for
$G_{(\{a_1,\dots,a_n\})}$.
If $G'$ is a subgroup of $G$, then $G'_{(A)}=G'\cap G_{(A)}$.
For $x\in\MM$, we say that
$A$ \emph{supports} $x$ if $xg=g$ for each $g\in G$ which fixes at $A$,
and $x$ is \emph{symmetric} if $x$ has a support which is 
a member of $\II$.
\end{notedefr*}

Let $\NN$ be the permutation model consisting
of hereditarily symmetric elements of $\MM$.
Note that the empty set supports the partition $\mathcal P$
of $U$ described above, and also supports any well-ordering of $\mathcal P$
in $\MM$.  So in $\NN$, $\mathcal P$ is a
countable partition of the set $U$ of atoms
into sets of size $p$.  
However, no choice function for 
$\mathcal P$ has a support in $\II$, and so
$\NN\models\lnot\text{\statec{p}}$.

\begin{remarks*}\label{remarks:supports}
\begin{enumerate}
\item[(1)]
Note, by ($*$2) above,
that $A$ supports $x$ if and only if $\Span(A)$ supports $x$,
and thus $A$ supports $x$ if and only if any basis for $\Span(A)$ supports
$x$.
\item[(2)] Let $w\in W$.  Observe that
for any $g\in G$, $g$ fixes one
element of $U_w$ if and only if $g$ fixes each element of $U_w$, and
$G_{(w)}$ is the stabilizer subgroup of each element of $U_w$.
\end{enumerate}
\end{remarks*}

We now want to show that
$\NN\models\text{CP($\aleph_0$,$p$)}$. We first establish a couple of
lemmas about supports of elements of $\NN$.

\begin{lemma}\label{lemma:finitesupport}
Suppose $A\in\II$ and $x\in\NN$.
Either there is a finite set $B\subset W$ such that $B\cup A$ supports
$x$, or the 
$G_{(A)}$-orbit of $x$ is 
infinite.
\end{lemma}
\begin{proof}
We give a forcing argument similar to one used in Shelah \cite{Sh:273}.
We set up a notion of forcing $\QQ$ which
adds a new automorphism of $U$ like those found in $G_{(A)}$.  
Let $A^\perp$ be a subspace of $W$ complementary to $A$
(i.e., $\Span(A\cup A^\perp)=W$ and $A\cap A^\perp=\{\mathbf{0}\}$),
and fix a basis   
$\setof{w_i}{i\in\omega}$ for $A^\perp$.  
Conditions of $\QQ$ shall have the following
form:  For any $n\in\omega$ and function $f\colon n\to\FF_p$, let
$q_f$ be the unique automorphism of $\FF_p\times\Span\{w_0,\dots w_{n-1}\}
\subset U$ which fixes each $U_{w_i}$ and
maps $(0,w_i)$ to $(f(i),w_i)$.  As usual, for conditions $q_1,q_2\in\QQ$,
we let $q_1\le q_2$ iff $q_2\subseteq q_1$.  Thus, if $\Gamma\subset\QQ$
is a generic filter, then $\pi=\bigcup\Gamma$ is an automorphism of
$A^\perp$ preserving the partition $\mathcal{P}$.
Easily, $\pi$ extends uniquely
to an automorphism of $U$
fixing at $A$ and preserving the partition $\mathcal{P}$,
and thus we will think of such $\pi$ as being an automorphism of $U$.
Observe that $\QQ$ is equivalent to Cohen forcing 
(the way we have associated each condition with a finite sequence
of elements of 
$\FF_p$, it is easy to think
of $\QQ$ as just adding a Cohen generic sequence 
in ${}^\omega\FF_p$).
Let $\dot\pi$ be a canonical name for the automorphism added by $\QQ$.
Let $(\QQ_1,\dot\pi_1)$ and
$(\QQ_2,\dot\pi_2)$ each be copies of $(\QQ,\dot\pi)$.

\textsc{Case 1:} For some $(q_1,q_2)\in\QQ_1\times\QQ_2$, 
                 $(q_1,q_2)\forces \check x\dot\pi_1 = \check x\dot\pi_2$.

Let $B\subset W$ be some finite support for $q_1$;
e.g.\ $B=\setof{w\in W}{(\exists n\in\FF_p)\ (n,w)\in\Dom(q_1)\cup\Ran(q_1)}$.
Let $\Gamma\subset\QQ_1\times\QQ_2$ be
generic over $\MM$ with $(q_1,q_2)\in\Gamma$,
and let $(\pi_1,\pi_2)$ be the interpretation of
$(\dot\pi_1,\dot\pi_2)$ in $\MM[\Gamma]$.  
For any $g\in G_{(A\cup B)}$, $(g\pi_1,\pi_2)$ is another
$\QQ_1\times\QQ_2$\,-generic pair of automorphisms. 
Let $\Gamma_g\subset\QQ_1\times\QQ_2$ such that 
$(g\pi_1,\pi_2)$ is the interpretation of $(\dot\pi_1,\dot\pi_2)$
in $\MM[\Gamma_g]$.

Note that $(q_1,q_2)$ is in both $\Gamma$ and $\Gamma_g$, so
$\MM[\Gamma]\models x\pi_1=x\pi_2$, and $\MM[\Gamma_g]\models xg\pi_1=x\pi_2$.
Thus, $x\pi_1=xg\pi_1$ (if desired, one can briefly reason
in an extension which contains both $\Gamma$ and $\Gamma'$), and 
it follows that $x=xg$ (recall that 
automorphisms of $U$
which preserve $\mathcal P$,
such as $g$ and $\pi_1$,
commute).

We have shown that every $g\in G_{(A\cup B)}$ fixes $x$, which 
is to say that $A\cup B$ supports $x$, which completes the proof for Case 1.

\textsc{Case 2:} $\forces_{\QQ_1\times\QQ_2}
                 \check x\dot\pi_1 \neq \check x\dot\pi_2$.

Let $\mathcal{H}(\kappa)$ be the set of hereditarily 
of cardinality smaller than
$\kappa$ sets, where $\kappa>2^{\aleph_0}+|\mathrm{TC}(x)|$,
and let $C$ be
a countable elementary submodel of $\mathcal{H}(\kappa)$ with $x\in C$.
It is clear that there exist infinitely many elements
of $G_{(A)}$ which are mutually
$\QQ$\,-generic over $C$, and in fact there is perfect set
such elements by \cite{Sh:273} (specifically, Lemma 13, applied
to the equivalence relation $\mathcal E$ on $G_{(A)}$ defined by 
$\pi_1\mathrel{\mathcal E}\pi_2  \leftrightarrow x\pi_1=x\pi_2$).
More precisely, there is a perfect set $P\subset G_{(A)}$ such
that for each $\pi_1,\pi_2\in P$, $(\pi_1,\pi_2)$ is
$\QQ_1\times\QQ_2$\,-generic over $C$. Thus $x\pi_1\neq x\pi_2$
whenever $\pi_1,\pi_2\in P$, and 
hence, the $G_{(A)}$-orbit of $x$ is infinite.
\end{proof}

\begin{lemma}\label{lemma:singletonsupport}
Let $x\in X\in\NN$.  Let $A\in\II$ be a support for
$X$.  If $|X|=p$, then there exists $b\in W$ such that $A\cup\{b\}$ supports
$x$.
\end{lemma}

\begin{proof}
Since $G_{(A)}$ fixes $X$, the $G_{(A)}$ orbit of $x$ is contained
in $X$, and hence is finite.  By the previous lemma, there is
a finite set $B\subset W$ such that $A\cup B$ supports
$x$.  We will show that if $|B|>1$, then there is some $B'$ with
$|B'|<|B|$ such that $A\cup B'$ supports $x$; the lemma then follows
by induction.
Assume, without loss of generality, that
\[\text{$B$ is a linearly independent
set disjoint from $\Span(A)$,} \tag{$*$}
\]
and let $B=\{b_1,\dots,b_n\}$, where this is a set of $n$ distinct
elements.
Assume also that for each proper subset $C\subset B$, $A\cup C$
fails to support $x$ (otherwise we are done easily).
Then 
$G_{(A\cup C)}$ acts non-trivially on $X$ for each proper $C\subset B$,
and, since
$G_{(A\cup C)}$ is a $p$-group and $|X|=p$,
the $G_{(A\cup C)}$-orbit of $x$ must be $X$.   
Let $G'=G_{(A\cup\{b_3,\dots,b_n\})}$.
Let $H$ be the stabilizer of
$x$ in $G'$; that is, $H=\setof{g\in G'}{xg=x}$.
Then $[G':H]=|\Orb_{G'}(x)|=p$.
Note that $G'_{(b_1,b_2)}=G_{(A\cup B)}\subset H$ since $A\cup B$
supports $x$.  It suffices to find $b\in W$ such that
$G'_{(b)}\subseteq H$, for then $A\cup\{b,b_3,\dots,b_n\}$ supports
$x$, as desired.

Recall (by Remark (2) above) that $G'_{(b_1,b_2)}$ is the stabilizer
in $G'$
of any ordered pair $(u_1,u_2)\in U_{b_1}\times U_{b_2}$.  
It follows from ($*$) that there exist elements of $G$ which move the
$p$ elements of $U_{b_1}$ while fixing all elements of $U_{w}$ for
each $w\in A\cup\{b_2,b_3,\dots,b_n\}$,
and likewise with $b_1$ and $b_2$ switched.
Thus $U_{b_1}\times U_{b_2}$ itself
is the $G'$-orbit of the pair $(u_1,u_2)$, and
so
$[G':G'_{(b_1,b_2)}]=p^2$.
Therefore, $[H:G'_{(b_1,b_2)}]=p$.

Let $h\in H\setmin G'_{(b_1,b_2)}$. The natural
image of $h$ in the quotient group
$H/G'_{(b_1,b_2)}$ generates that quotient group (which has order $p$),
and therefore
$h$ generates the action of $H$ on $U_{b_1}\times U_{b_2}$.
Let $m,n\in\FF_p$ such that
\[
(0,b_1)h=(m,b_1)\quad\text{and}\quad
(0,b_2)h=(n,b_2).
\]
Note that if $m=0$ or $n=0$, then we are done easily:  Say $m=0$. 
Then $h$ fixes at $b_1$, and consequently every element of $H$ 
fixes at $b_1$, so $H\subseteq G'_{(b_1)}$.
But then $H=G'_{(b_1)}$, since both subgroups have the same index
in $G'$, and the proof is completed by taking $b=b_1$.

On the other hand, if $m$ and $n$ are both nonzero, then
we have inverses $m^{-1}$ and
$n^{-1}$ in $\FF_p$, and we let $b=m^{-1}b_1-n^{-1}b_2$.
Now we just want to show that $G'_{(b)}\subseteq H$.  But since
these two groups have the same index in $G'$, it is equivalent to show
$H\subseteq G'_{(b)}$.  Compute:
\begin{multline*}
(0,b)h=
(0,m^{-1}b_1-n^{-1}b_2)h=m^{-1}(0,b_1)h-n^{-1}(0,b_2)h= \\
m^{-1}(m,b_1)-n^{-1}(n,b_2)= 
(1,m^{-1}b_1)-(1,n^{-1}b_2)=(0,b).
\end{multline*}
Thus $h$ fixes at $\{b\}$, and so does every power of $h$.  Since
every element of $H$ acts on $b_1$ and $b_2$ like a power of $h$, it
follows that $H\subseteq G'_{(b)}$, as desired.
\end{proof}

Now, to show 
$\NN\models\text{CP($\aleph_0$,$p$)}$, let $Z=\setof{X_n}{n\in\omega}$
be a set of $p$-elements sets, with $Z$ countable in $\NN$. 
Let $A\in\II$ be a support for a well-ordering of $Z$, so that
$A$ is a support for each element of $Z$.
For each $n\in\omega$, let $x_n\in X_n$
(of course, $Z$ might not have a choice function in $\NN$, but we
are working in $\MM$).  By Lemma~\ref{lemma:singletonsupport}, 
since $|X_n|=p$,
there is some $s_n\in W$ such that
$A\cup\{s_n\}$ supports $x_n$.
Let $S=\setof{s_n}{n\in\omega}$.  
If $S$ is finite, then $A\cup S\in\II$, and $A\cup S$ is a support
for the enumeration $\langle x_n\rangle_{n\in\omega}$,
so in fact $Z$ has a choice function in $\NN$.
If $S$ is infinite, then there is an infinite $B\in\II$ such
that $B\subseteq S$; say $B=\setof{s_n}{n\in J}$.  In this case,
$A\cup B$ is a support for the enumeration 
$\langle x_n\rangle_{n\in J}$.
In either case, 
$Z$ has an infinite partial choice function in $\NN$, as desired.
\end{proof}

It remains in this section to establish the existence of an ideal on 
$W=\oplus_{i=0}^\infty\FF_p$ having the properties needed in the
proof of Theorem~\ref{theorem:main}.  

\smallskip
\noindent
\textit{Notation and definitions.}
\begin{enumerate}
\setlength{\itemsep}{0ex}
\item [1.] For $n\in\omega\setmin\{0\}$,
            let $\logp(n)$ be the least $k\in\omega$ such that
            $(\log_p)^k(n)\le 1$, where $(\log_p)^0(n)=n$ and
            $(\log_p)^{k+1}(n)=\log_p\left((\log_p)^k(n)\right)$.

 Note: In what follows, $\logp$ could be replaced by 
       $\log_{*}=\log_{*2}$ with no effect on the arguments,
       except for minor changes needed in part (2) of 
       Lemma~\ref{lemma:ideal}.
\item [2.] For convenience,
            let $\setof{e_k}{k\in\omega}$ be the ``standard basis''
            for $W=\oplus_{i=0}^\infty \FF_p$; i.e.
           \[
           e_k(i)=\begin{cases}1 & \text{if $i=k$},\\0 &\text{else.}\end{cases}
           \]
\item [3.] For $\displaystyle w=\sum_{\ell}a_\ell e_\ell\in W$,
let $\displaystyle \pr_k(w)=\sum_{\ell<k}a_\ell e_\ell$.  
\item [4.] $\dd_k(A)=|\setof{\pr_k(w)}{w\in A}|$.
\item [5.] We say $A\subset W$ is \emph{thin} if
            \[
            \lim_{k\to\infty}\frac{\logp[\dd_k(A)]}{\logp(k)}\ =\ 0.
            \]
\end{enumerate}

\begin{lemma}\label{lemma:ideal}
Let $\II$ be the set of thin subsets of $W$.  Then 
\begin{enumerate}
\item[(0)] $\II$ is an ideal on $W$,
\item[(1)] every infinite subset of $W$ contains an infinite member of
           $\II$, and  
\item[(2)] $A\in\II\Rightarrow\Span(A)\in\II$.
\end{enumerate}
\end{lemma}
\begin{proof}
(0)  Clearly $\II$ is closed under subset.  Suppose $A_1$ and $A_2$
are thin, and let $A=A_1\cup A_2$.  Then 
(for any $k\in\omega$) $\dd_k(A)\le\dd_k(A_1)+\dd_k(A_2)$,
so
\[
\frac{\logp[\dd_k(A)]}{\logp(k)} \le 
\frac{\logp[\dd_k(A_1)+\dd_k(A_2)]}{\logp(k)} \le
\frac{1+\max_{i=1,2}\left(\logp[\dd_k(A_i)]\right)}{\logp(k)}.
\]
The limit as $k\to\infty$ must be $0$, so $A$ is thin.

(1)  Let $A\subseteq W$ be an infinite thin set.  By K\"onig's Lemma, 
we can find pairwise distinct $x_n\in A$ for $n\in\omega$ such that 
for each $i\in\omega$, $\sq{x_n(i)}_{n<\omega}$ is eventually constant.

Let $n_0=0$.  For $i\in\omega$, assuming $n_0,\dots n_i$ are chosen,
we can choose $n_{i+1}$ large enough so that
\begin{align*}
 & \pr_{n_i}(x_{n_{i+1}})=\pr_{n_i}(x_{n_{i+1}+t})\quad\text{for all $t\in\omega$}, \\
\text{and}\quad & \logp(n_{i+1})>i+1.
\end{align*}
Let $A^-=\setof{x_{n_i}}{i\in\omega}$.
Then $\dd_{n_i}(A^-)\le i+1$, and
\[
\lim_{i\to\infty}\frac{\logp(\dd_{n_i}(A^-))}{\logp(n_i)} \le
\lim_{i\to\infty}\frac{\logp(i+1)}{i} = 0.
\]
Therefore $A^-$ is an infinite, thin subset of $A$.

(2) 
For any $A\subset W$, observe that
\[
\dd_k(\Span A)\le p^{\dd_k(A)}.
\]
Thus 
\[
\logp(\dd_k(\Span A)) \le \logp\left(p^{\dd_k(A)}\right)\le
1+\logp(\dd_k(A)).
\]
It follows easily that if $A$ is thin, then $\Span A$ is also thin.
\end{proof}

Everything needed for Theorem~\ref{theorem:main} has now been proven.
Note that this theorem does not say anything about the
independence of \statec{n}\ from \statepc{n}\ in ZF when $n$ is
not prime.  We intend to consider the case when $n$ is not prime
elsewhere.

\end{section}
\begin{section}{Simpler models 
not useful for the main theorem} \label{simpler}

We consider a family of permutation models, some of which may on first
consideration seem to be promising candidates to witness
that \statepc{2} $\notarrow$ \statec{2}.  However, it will turn out that
\statepc{2} fails in every such model.

Let $\MM$ be a model of ZFAC whose set $U$ of atoms is countable and
infinite.  Let $\mathcal P=\setof{U_n}{n\in\omega}$ be a partition of 
$U$ into pairs.   Let $G$ be the group of permutations of $U$ (equivalently,
automorphisms of $\MM$) which fix each element of $\mathcal P$.  
Let $\II$ be some ideal on $\omega$.  For $A\in\II$ and $g\in G$, we say
$g$ \emph{fixes at} $A$ if $g$ fixes each element of
$\bigcup_{n\in A} U_n$.  Define \emph{support} and \emph{symmetric}
by analogy with the definitions of these terms
in the proof of the main theorem, and let $\NN$ be the permutation
submodel consisting of the hereditarily symmetric elements. 

If $\II$ is the ideal of finite subsets of $\omega$, then $\NN$ is
the ``second Fraenkel model.''  Clearly $\mathcal P$ has no
infinite partial choice function in the second Fraenkel model. 
Of course,
if $\II$ is any larger than the finite set ideal, then 
$\mathcal P$ does have an infinite partial choice 
function, and
it may be tempting to think that if $\II$ is well-chosen, then
perhaps \statepc{2} will hold in
the resulting model $\NN$.  However, we
will show how to produce a set $Z=\setof{X_n}{n\in\omega}$ of
pairs, countable in $\NN$, with no infinite partial choice 
function (no matter how $\II$ is chosen).

Notation: For sets $A$ and $B$, let $P(A,B)$ be the set of bijections
from $A$ to $B$.  We are interested in this when $A$ and $B$ are both 
pairs, in which case $P(A,B)$ is also a pair.

Let $X_0=A_0$.  For $i\in\omega$, let $X_{i+1}=P(X_i,A_{i+1})$.  
The empty set supports each pair $X_i$, so $Z=\setof{X_n}{n\in\omega}$ is a 
countable set in $\NN$.  
Let $S\in\II$;
we'll show that $S$ fails to support any infinite partial choice
function for $Z$. 
Let $i=\min(\omega\setmin S)$, and 
let $g\in G$ be the permutation which swaps the elements of $A_i$ and
fixes all other atoms, so $g\in G_{(S)}$. 
This $g$ fixes each element of $X_n$ for $n<i$, 
but swaps the elements of $X_i$.  By simple induction,
$g$ also swaps the elements of $X_n$ for all $n>i$.  
It follows that for any $C\in\MM$ which is an infinite partial
choice function for $Z$, $Cg\neq C$, and thus $S$ does not
support $C$.

\end{section}

\bibliographystyle{plain}
\bibliography{934,listb}

\end{document}